\documentclass{article}
\usepackage{amssymb}
\usepackage{latexsym}
\usepackage{amsmath}
\usepackage{mathrsfs}
\usepackage{color}

\newtheorem{theorem}{Theorem}[section]

\newtheorem{lemma}{Lemma}[section]
\newtheorem{proposition}{Proposition}[section]
\newtheorem{remark}{Remark}[section]

\numberwithin{equation}{section}
\newenvironment{proof}{\medskip\par\noindent{\bf Proof.}\ }{\qquad
\raisebox{-0.5mm}{\rule{1.5mm}{4mm}}\vspace{6pt}}

\newcommand{\bbr}{\mathbb{R}}
\newcommand{\h}{H^1(\bbr^3)}
\newcommand{\D}{D^{1,2}(\bbr^3)}

\newcommand{\bbn}{\mathbb{N}}
\newcommand{\ve}{\varepsilon}

\begin{document}

\title
{\Large\bf On the Schr\"odinger-Poisson system with steep potential well and indefinite potential}

\author{Juntao Sun$^{a},$\thanks{E-mail address: sunjuntao2008@163.com(Juntao Sun)}\quad
Tsung-fang Wu$^b,$\thanks{E-mail address:  tfwu@nuk.edu.tw (Tsung-fang Wu)}
\quad
Yuanze Wu$^c$\thanks{Corresponding author.  E-mail address:  wuyz850306@cumt.edu.cn (Yuanze Wu)}\\
\footnotesize$^a${\em School of Science, Shandong University of Technology, Zibo, 255049, P.R. China }\\
\footnotesize$^b${\em  Department of Applied Mathematics, National University of Kaohsiung, Kaohsiung 811, Taiwan }\\
\footnotesize$^c${\em  College of Sciences, China
University of Mining and Technology, Xuzhou 221116, P.R. China  }}
\date{}
\maketitle

\date{}
\maketitle

\noindent{\bf Abstract:} In this paper, we study the following Schr\"odinger-Poisson system:
$$
\left\{\aligned&-\Delta u+V_\lambda(x)u+K(x)\phi u=f(x,u)&\quad\text{in }\bbr^3,\\
&-\Delta\phi=K(x)u^2&\quad\text{in }\bbr^3,\\
&(u,\phi)\in\h\times\D,\endaligned\right.\eqno{(\mathcal{SP}_{\lambda})}
$$
where $V_\lambda(x)=\lambda a(x)+b(x)$ with a positive parameter $\lambda$, $K(x)\geq0$ and $f(x,t)$ is continuous including the power-type nonlinearity $|u|^{p-2}u$.  By applying the method of penalized functions, the existence of one nontrivial solution for such system in the less-studied case $3<p\leq4$ is obtained for $\lambda$ sufficiently large.  The concentration behavior of this nontrivial solution for $\lambda\to+\infty$ are also observed.  It is worth to point out that some new conditions on the potentials are introduced to obtain this nontrivial solution and the Schr\"odinger operator $-\Delta+V_\lambda(x)$ may be strong indefinite in this paper.

\vspace{6mm} \noindent{\bf Keywords:} Schr\"odinger-Poisson system; Steep potential well; Strong indefinite problem.


\section{Introduction}
In this paper, we consider the following Schr\"odinger-Poisson system:
$$
\left\{\aligned&-\Delta u+V_\lambda(x)u+K(x)\phi u=f(x,u)&\quad\text{in }\bbr^3,\\
&-\Delta\phi=K(x)u^2&\quad\text{in }\bbr^3,\\
&(u,\phi)\in\h\times\D,\endaligned\right.\eqno{(\mathcal{SP}_{\lambda})}
$$
where $V_\lambda(x)=\lambda a(x)+b(x)$ with a positive parameter $\lambda$, $K(x)\geq0$ is a weight function and $f(x,t)$ is a continuous nonlinearity.

The Schr\"odinger-Poisson system~$(\mathcal{SP}_{\lambda})$ arises in quantum mechanic.  It was first introduced in \cite{BF98} as a physical model describing a charged wave interacting with its own electrostatic field.  In such system, the unknowns $u$ and $\phi$ represent the wave functions associated to the particle and electric potential, and functions $V_\lambda(x)$ and $K(x)$ are respectively an external potential and nonnegative density charge.  This system also arises in the electromagnetic field, semiconductor theory, nonlinear optics and plasma physics and sometimes it is called as the Scor\"odinger-Maxwell system.  Due to the important applications in physics, there has been a vast literature on the study of the existence and multiplicity for nontrivial solutions to the system~$(\mathcal{SP}_{0})$ under various hypotheses on the potentials $b(x)$ and $K(x)$ and the nonlinearities $f(x,t)$.  We refer the readers to \cite{AD08,A10,CKW13,CV10,I13,IR12,K07,LWZ14,R06,R10,SCJ12,WZ09} and the references therein.

In this paper, we will investigate the Schr\"odinger-Poisson system~$(\mathcal{SP}_{\lambda})$ for $\lambda>0$ and assume $a(x)$ satisfies the following conditions:
\begin{enumerate}
\item[$(A_1)$] $a(x)\in C(\bbr^3, \bbr)$ and $a(x)$ is bounded below.
\item[$(A_2)$] There exists $a_\infty>0$ such that $|\mathcal{A}_\infty|<+\infty$, where $\mathcal{A}_\infty=\{x\in\bbr^3\mid a(x)<a_\infty\}$ and $|\mathcal{A}_\infty|$ is the Lebesgue measure of the set $\mathcal{A}_\infty$.
\item[$(A_3)$] $\Omega_a=$int$a^{-1}(0)$ is a nonempty open set and has smooth boundary with $\overline{\Omega}_a=a^{-1}(0)$.
\item[$(A_4)$] dist$(\Omega_a,\bbr^3\backslash\mathcal{A}_\infty)>0$, where dist$(\Omega_a,\bbr^3\backslash\mathcal{A}_\infty)=\inf\{|x-y|\mid x\in\Omega_a,y\in\bbr^3\backslash\mathcal{A}_\infty\}$.
\end{enumerate}
\begin{remark}
If $\Omega_a$ is bounded, then by the condition $(A_1)$, the assumption $(A_4)$ is trivial.  However, under the conditions $(A_1)$-$(A_4)$, $\Omega_a$ is allowed to be unbounded.  Furthermore, $a(x)$ is allowed to be sign-changing under the conditions $(A_1)$--$(A_4)$.
\end{remark}
Under the conditions $(A_1)$--$(A_3)$, $\lambda a(x)$ is called as the steep potential well for $\lambda$ sufficiently large and the depth of the well is controlled by the parameter $\lambda$.  Such potentials were first introduced by Bartsch and Wang in \cite{BW95} for the scalar Schr\"odinger equations.  An interesting phenomenon for this kind of Schr\"odinger equations is that, one can expect to find the solutions which are concentrated at the bottom of the wells as the depth goes to infinity.  Due to this interesting property, such topic for the scalar Schr\"odinger equations was studied extensively in the past decade.  We refer the readers to \cite{BW00,BPW01,BT13,DT03,DS07,LHL11,ST09,WZ09} and the references therein.  Recently, the steep potential well was also considered for some other elliptic equations and systems, see for example \cite{FSX10,GT121,SW14,SW141,YT14,ZLZ13} and the references therein.  In particular, the steep potential well was introduced to the Schr\"odinger-Poisson system by Jiang and Zhou in \cite{JZ11}, where System~$(\mathcal{SP}_{\lambda})$ with $a(x)\geq0$, $b(x)=K(x)=1$ and $f(x,u)=|u|^{p-2}u$ are considered.  They obtained the existence results of one nontrivial solution for the case $p\in (2,3)\cup \lbrack 4,6)$ if $\lambda$ is sufficiently large and the concentration behavior of the nontrivial solutions for $\lambda\to+\infty$ was also observed in their paper.  System~$(\mathcal{SP}_{\lambda})$ with $b(x)=0$ and $f(x,u)=|u|^{p-2}u$ if $a(x)$ satisfies the conditions $(A_1)$--$(A_3)$ was also investigated by Zhao et al. in \cite{ZLZ13}, where one nontrivial solution is founded for the case $p\in (3, 6)$.  In the case of $3<p<4$, the condition $a(x)\geq0$ was also needed in \cite{ZLZ13} and the concentration behavior of the nontrivial solutions for $\lambda\to+\infty$ was also obtained in this case.  Zhao et al.'s results for the case $4<p<6$ was generalized by Ye and Tang in \cite{YT14} in the sense that a more general nonlinearity $f(x,u)$ is considered.  Ye and Tang also obtained the existence and multiplicity of nontrivial solutions for the System~$(\mathcal{SP}_{\lambda})$ if the nonlinearity $f(x,u)$ is sub-linear.  In the very recent work \cite{SW141}, the first and second authors of the current paper study the System~$(\mathcal{SP}_{\lambda})$ with $a(x)\geq0$ when $f(x,u)$ is asymptotically linear, 3-asymptotically linear and 4-asymptotically linear, where some existence and nonexistence results for nontrivial solutions were obtained and the concentration behavior of the nontrivial solutions for $\lambda\to+\infty$ was also observed.  These results partially complement the study on the System~$(\mathcal{SP}_{\lambda})$.

We point out that the case $3<p<4$ for the Schr\"odinger-Poisson System~$(\mathcal{SP}_{\lambda})$ seems to be more interesting and difficult than other cases under the point of calculus of variations, since it is hard to obtain the boundness of every (PS) sequence in this case.  Due to this reason, the Schr\"odinger-Poisson System~$(\mathcal{SP}_{\lambda})$ in the case of $3<p<4$ are less-studied by the variational methods.  In \cite{ZZ08}, some Pohozaev type conditions on the potentials were introduced to get a special bounded (PS) sequence of the Schr\"odinger-Poisson System~$(\mathcal{SP}_{\lambda})$ in the case $3<p<4$ by the monotonicity trick of Jeanjean \cite{J99}.  Similar Pohozaev type conditions on the potentials were also used in \cite{LWZ14,YD10,ZLZ13} for this case.  Inspired by the above facts, we wonder whether the Pohozaev type conditions on the potentials are necessary for finding nontrivial solutions of $(\mathcal{SP}_{\lambda})$ in the case of $3<p<4$?  In the current paper, we will explore this question.

Let us first list out our assumptions on $b(x)$, $K(x)$ and $f(x,t)$:
\begin{enumerate}
\item[$(B_1)$] $b(x)\in C(\bbr^3, \bbr)$.
\item[$(B_2)$] There exists $b_0>0$ such that $b^-(x)\leq b_0(1+\sqrt{a(x)})$ for all $x\in\bbr^3$, where $b^-(x)=\max\{0, -b(x)\}$.
\item[$(B_3)$] Either $b(x)\geq0$ or $b^\pm(x)\not\equiv0$ on $\Omega_a$ with $|\mathcal{B}_0|<+\infty$, where $b^+(x)=\max\{b(x), 0\}$ and $\mathcal{B}_0=\{x\in\bbr^3\mid b(x)<0\}$.
\item[$(K_1)$] $K(x)\in L^2(\bbr^3)$.
\item[$(K_2)$] $\Omega_K=$int$K^{-1}(0)$ is a nonempty open set and dist$(\Omega_a,\bbr^3\backslash\Omega_K)>0$.
\item[$(F_1)$] There exist $p\in(2, 6)$, $f_0>0$ and $f_1>0$ such that $-f_1(|t|^2+|t|^p)\leq f(x,t)t\leq f_0(|t|^2+|t|^{p})$ for all $x\in\bbr^3$ and $t\in\bbr$.
\item[$(F_2)$] $\lim_{t\to0}\frac{f(x,t)}{t}=0$ uniformly for $x\in\bbr^3$.
\item[$(F_3)$] $\lim_{t\to\infty}\frac{f(x,t)}{t^2}=+\infty$ uniformly for $x\in\bbr^3$.
\item[$(F_4)$] There exists $f_2>0$ such that $f(x,t)t-2F(x,t)\geq f_2|t|^{3}$ for all $x\in\bbr^3$, where $F(x,t)=\int_0^tf(x,s)ds$.
\item[$(F_5)$] There exist $f_3>0$ and $\mu>2$ such that $f(x,t)t-\mu F(x,t)\geq-f_3|t|^3$ for all $x\in\bbr^3$ and $t\in\bbr$.
\end{enumerate}
\begin{remark}
\begin{enumerate}
\item[$(1)$] If $b(x)$ is bounded below, then the condition $(B_2)$ is trivial.  But under the assumptions $(B_1)$-$(B_3)$, $b(x)$ may be unbounded from below and sign-changing.
\item[$(2)$] A typical function satisfying the conditions $(F_1)$-$(F_5)$ is that $f(x,t)=|t|^{p-2}t$ with $3<p<6$.  However, $f(x,t)$ may be sign-changing under the conditions $(F_1)$-$(F_5)$.
\end{enumerate}
\end{remark}
Now, the main result in this paper can be stated as follows.
\begin{theorem}\label{thm0001}
Suppose the conditions $(A_1)$-$(A_4)$, $(B_1)$-$(B_3)$, $(K_1)$-$(K_2)$ and $(F_1)$-$(F_5)$ hold.  If $a(x)\geq0$, $3<p\leq4$ and $0\not\in\sigma(-\Delta+b(x), H_0^1(\Omega_a))$, then there exists $\Lambda_*>0$ and $f_1^*>0$ such that $(\mathcal{SP}_{\lambda})$ has a nontrivial weak solution $(u_\lambda,\phi_\lambda)$ for $\lambda\geq\Lambda_*$ and $f_1<f_1^*$, where $\sigma(-\Delta+b(x), H_0^1(\Omega_a))$ is the spectrum of $-\Delta+b(x)$ on $H_0^1(\Omega_a)$.  Furthermore, for every $\{\lambda_n\}$, $(u_{\lambda_n},\phi_{\lambda_n})\to (u_0,\phi_0)$ strongly in $\h\times D^{1,2}(\bbr^3)$ as $n\to\infty$, where $\phi_0(x)=\int_{\Omega_a}\frac{K(y)u_0^2(y)}{4\pi|x-y|}dy$ and $u_0$ is a nontrivial weak solution of the following equation:
\begin{equation}  \label{eq0022}
-\Delta u+b(x)u+K(x)\bigg(\int_{\Omega_a}\frac{K(y)u^2(y)}{4\pi|x-y|}dy\bigg) u=f(x,u),\quad u\in H^1_0(\Omega_a).
\end{equation}
\end{theorem}

\begin{remark}
In Theorem~\ref{thm0001}, the conditions $(A_4)$ and $(K_2)$ are introduced to replace the Pohozaev type conditions on the potentials, which seems to be new for the existence of nontrivial weak solutions to the Schr\"odinger-Poisson system $(\mathcal{SP}_{\lambda})$ with the steep potential well in the case of $3<p\leq4$.  Moreover, under the conditions $(A_1)$--$(A_4)$ and $(B_1)$--$(B_3)$, the Schr\"odinger operator $-\Delta+V_\lambda(x)$ may be strong indefinite if $\lambda$ sufficiently large (see Lemma~\ref{lem0002} for more details), and it seems to be the first time that such a Schr\"odinger operator is considered for the Schr\"odinger-Poisson system~ $(\mathcal{SP}_{\lambda})$ in the case of $3<p\leq4$.
\end{remark}

The remaining of this paper is devoted to the proof of Theorem~\ref{thm0001} and it will be organized as follows.  In section~2, we will introduce an auxiliary system of the Schr\"odinger-Poisson system $(\mathcal{SP}_{\lambda})$ by the method of penalized functions and obtain one nontrivial solution of the auxiliary system by the well-known linking theorem.  In section~3, by the Morse iteration, we will prove that some special nontrivial solutions of the auxiliary system are also the nontrivial solutions of the Schr\"odinger-Poisson system $(\mathcal{SP}_{\lambda})$ for $\lambda$ sufficiently large.  Theorem~\ref{thm0001} will also be shown in this section.

In this paper, we will always denote the usual norms in $\h$ and $L^p(\bbr^3)$ ($p\geq1$) by $\|\cdot\|$ and $\|\cdot\|_{L^p(\bbr^3)}$, respectively.  $o_n(1)$ will always denote the quantities tending towards zero as $n\to\infty$.

\section{The auxiliary system}
In this section, we will introduce an auxiliary system of $(\mathcal{SP}_{\lambda})$ by the method of penalized functions.  This method was developed by del Pino and Felmer in \cite{DF96} and also used to find nontrivial solutions of elliptic equations or systems with the steep potential well in several other literatures, see for example \cite{BT13,DT03,GT121,ST09} and the references therein.  For the sake of convenience, we always assume the conditions $(A_1)$--$(A_4)$, $(B_1)$--$(B_3)$, $(K_1)$--$(K_2)$ and $(F_1)$--$(F_5)$ hold with $a(x)\geq0$ in this section.

Since the conditions $(A_4)$ and $(K_2)$ hold, we can choose $\Omega_a'$ to be a nonempty open set in $\bbr^3$ which satisfies $\Omega_a\subset\Omega_a'\subset(\mathcal{A}_\infty\cap\Omega_K)$, dist$(\Omega_a, \bbr^3\backslash\Omega_a')>0$ and dist$(\Omega_a', \bbr^3\backslash(\mathcal{A}_\infty\cap\Omega_K))>0$.  For $\delta\in(0, \min\{1,\frac12(S^{-1}\|K\|_{L^2(\bbr^3)}^{-2}f_2)^{\frac12}\})$, we define
\begin{equation*}
g_\delta(s)=\left\{\aligned &-\delta,&\quad s\leq-\delta,\\
&s,&\quad |s|\leq\delta,\\
&\delta,&\quad s\geq\delta,\endaligned\right.
\end{equation*}
and $g_\delta(x,t)=2K(x)(\chi_{\Omega_a'}t+(1-\chi_{\Omega_a'})g_\delta(t))$, where $\chi_{\Omega_a'}$ is the characteristic function of $\Omega_a'$.
Let us consider the following elliptic system
$$
\left\{\aligned&-\Delta u+V_\lambda(x)u+\frac12\phi g_\delta(x,u)=f(x,u)&\quad\text{in }\bbr^3,\\
&-\Delta \phi=G_\delta(x,u)&\quad\text{in }\bbr^3,\\
&(u,\phi)\in\h\times\D,\endaligned\right.\eqno{(\overline{\mathcal{SP}}_{\lambda,\delta})}
$$
where $G_\delta(x,u)=\int_0^ug_\delta(x,t)dt$.  Clearly, if $(u,\phi)$ is a weak solution of $(\overline{\mathcal{SP}}_{\lambda,\delta})$ satisfying $|u(x)|\leq\delta$ a.e. on $\bbr^3\backslash\Omega_a'$, then $(u,\phi)$ is also a weak solution of $(\mathcal{SP}_\lambda)$.  On the other hand, for every $u\in \h$, since $|G_\delta(x,u)|\leq K(x)u^2$ for all $x\in\bbr^3$ and $u\in\bbr$, by the condition $(K_1)$, we have
\begin{eqnarray}   \label{eq0002}
\bigg(\int_{\bbr^3}|G_\delta(x,u)|^{\frac65}dx\bigg)^{\frac56}\leq\bigg(\int_{\bbr^3}K(x)^{\frac65}u^{\frac{12}{5}}dx\bigg)^{\frac56}
\leq\|K\|_{L^2(\bbr^3)}\|u\|_{L^{6}(\bbr^3)}^2.
\end{eqnarray}
It follows from the Sobolev inequality that $G_\delta(x,u)\in L^{\frac65}(\bbr^3)$.  By \cite[Lemma~2.1]{B02}, the Poisson equation $-\Delta \phi=G_\delta(x,u)$ in $\bbr^3$ has a unique solution in $\D$, which is given by
\begin{equation*}
\phi_{u,\delta}(x)=\frac{1}{4\pi}\int_{\bbr^3}\frac{G_\delta(y,u(y))}{|x-y|}dy.
\end{equation*}
Thus, the elliptic system $(\overline{\mathcal{SP}}_{\lambda,\delta})$ can be deduced to the following single Schr\"odinger equation
$$
\left\{\aligned&-\Delta u+V_\lambda(x)u+\frac12\phi_{u,\delta} g_\delta(x,u)=f(x,u)&\quad\text{in }\bbr^3,\\
&u\in\h,\endaligned\right.\eqno{(\overline{\mathcal{S}}_{\lambda,\delta})}
$$
in the sense that $u\in\h$ is a solution of $(\overline{\mathcal{S}}_{\lambda,\delta})$ is equivalent to $(u, \phi_{u,\delta})\in\h\times\D$ is a solution of $(\overline{\mathcal{SP}}_{\lambda,\delta})$.

\subsection{Variational setting of $(\overline{\mathcal{S}}_{\lambda,\delta})$}
For every $\lambda>0$, we define $E_\lambda=\{u\in \D\mid \int_{\bbr^3}(\lambda a(x)+b^+(x))u^2dx<+\infty\}$, where $b^+(x)=\max\{b(x), 0\}$.  Since the conditions $(A_1)$ and $(B_1)$ hold, $E_\lambda$ is a Hilbert space with the following inner product
\begin{equation*}
\langle u,v \rangle_\lambda=\int_{\bbr^3}(\nabla u\nabla v+(\lambda a(x)+b^+(x))uv)dx.
\end{equation*}
The corresponding norm is given by $\|u\|_\lambda=\langle u,u \rangle_\lambda^{\frac12}$.  By the condition $(A_2)$ and the H\"older and Sobolev inequalities, for every $u\in E_\lambda$ and $\lambda\geq1$, we have
\begin{equation}\label{eq0023}
\int_{\bbr^3}u^2dx\leq\max\{|\mathcal{A}_\infty|^{\frac23}S^{-1}, a_\infty^{-1}\}\int_{\bbr^3}|\nabla u|^2+(\lambda a(x)+b^+(x))u^2dx.
\end{equation}
Hence, $E_\lambda$ is embedded continuously into $\h$ for $\lambda\geq1$.
\begin{lemma}\label{lem0001}
$(\overline{\mathcal{S}}_{\lambda,\delta})$ has a variational structure in $E_\lambda$ for every $\lambda\geq1$ and the corresponding functional is given by
\begin{equation*}
\mathcal{E}_{\lambda,\delta}(u)=\frac12\mathcal{D}_{\lambda}(u,u)+\frac14\int_{\bbr^3}\phi_{u,\delta}G_\delta(x,u)dx-\int_{\bbr^3}F(x,u)dx,
\end{equation*}
where $\mathcal{D}_{\lambda}(u,v)=\int_{\bbr^3}(\nabla u\nabla v+(\lambda a(x)+b(x))uv)dx$ and $F(x,u)=\int_0^uf(x,s)ds$.
\end{lemma}
\begin{proof}
The idea of this proof comes from \cite{R06}.  We first claim that $\mathcal{E}_{\lambda,\delta}(u)$ is well defined on $E_\lambda$ for every $\lambda\geq1$.  Indeed, for every $u\in E_\lambda$ with $\lambda\geq1$, $\mathcal{D}_{\lambda}(u,u)=\|u\|_\lambda^2+\int_{\bbr^3}b^-(x)u^2dx$.  By the condition $(B_2)$ and \eqref{eq0023}, we have
\begin{eqnarray}
\int_{\bbr^3}b^-(x)u^2dx&\leq& b_0\int_{\bbr^3}(1+\sqrt{a(x)})u^2dx\notag\\
&\leq&b_0(1+\sqrt{a_\infty})\int_{\mathcal{A}_\infty}u^2dx
+b_0(\frac{1}{a_\infty}+\frac{1}{\sqrt{a_\infty}})\int_{\bbr^3\backslash\mathcal{A}_\infty}a(x)u^2dx\notag\\
&\leq&b_*\int_{\bbr^3}|\nabla u|^2+V_\lambda(x)u^2dx,\label{eq0024}
\end{eqnarray}
where $b_*=\max\{b_0(1+\sqrt{a_\infty})d_*, b_0(\frac{1}{a_\infty}+\frac{1}{\sqrt{a_\infty}})\}$ and $d_*=\max\{|\mathcal{A}_\infty|^{\frac23}S^{-1}, a_\infty^{-1}\}$.  It follows that $\mathcal{D}_{\lambda}(u,u)$ is well defined on  $E_\lambda$ for every $\lambda\geq1$.  On the other hand, since the condition $(F_1)$ holds and $E_\lambda$ is embedded continuously into $\h$ for $\lambda\geq1$, by using the H\"older and Sobolev inequalities in a standard way, we can see that $\int_{\bbr^3}F(x,u)dx$ is also well defined on $E_\lambda$ for every $\lambda\geq1$.  It remains to show $\int_{\bbr^3}\phi_{u,\delta}G_\delta(x,u)dx$ is well defined on $E_\lambda$ for $\lambda\geq1$.  In fact, since $\phi_{u,\delta}$ is the unique solution of the following Poisson equation
\begin{equation*}
-\Delta \phi=G_\delta(x,u),\quad\phi\in\D,
\end{equation*}
we can see that $\|\nabla\phi_{u,\delta}\|_{L^2(\bbr^3)}^2=\int_{\bbr^3}\phi_{u,\delta}G_\delta(x,u)dx$.  Hence, by the H\"older and Sobolev inequalities, we have
\begin{equation*}
\|\nabla\phi_{u,\delta}\|_{L^2(\bbr^3)}\leq S^{-\frac12}\bigg(\int_{\bbr^3}|G(x,u)|^{\frac65}dx\bigg)^{\frac56},
\end{equation*}
where $S$ is the best Sovolev embedding constant and given by
\begin{equation*}
S=\inf_{u\in D^{1,2}(\bbr^3)\backslash\{0\}}\frac{\|\nabla u\|_{L^2(\bbr^3)}^2}{\|u\|_{L^6(\bbr^3)}^2}.
\end{equation*}
This implies that $|\int_{\bbr^3}\phi_{u,\delta}G_\delta(x,u)dx|\leq S^{-1}\bigg(\int_{\bbr^3}|G(x,u)|^{\frac65}dx\bigg)^{\frac53}$.  Now, by \eqref{eq0002}, we can obtain that
\begin{equation}      \label{eq0004}
|\int_{\bbr^3}\phi_{u,\delta}G_\delta(x,u)dx|\leq S^{-2}\|K\|_{L^2(\bbr^3)}^2\|u\|_{L^{6}(\bbr^3)}^4.
\end{equation}
Since $E_\lambda$ is embedded continuously into $\h$ for $\lambda\geq1$, by the Sobolev inequality, we can see that $\int_{\bbr^3}\phi_{u,\delta}G_\delta(x,u)dx$ is also well defined on $E_\lambda$ for $\lambda\geq1$.  We complete the proof by showing that $\mathcal{E}_\lambda(u)$ is the corresponding functional of $(\overline{\mathcal{S}}_{\lambda,\delta})$ in $E_\lambda$ for $\lambda\geq1$.  Indeed, let $\Psi_\delta(u)=\int_{\bbr^3}\phi_{u,\delta}G_\delta(x,u)dx$.  Then $\Psi_\delta(u)$ is well defined on $E_\lambda$ for $\lambda\geq1$.  For every $u,v\in E_\lambda$, since the condition $(K_1)$ holds, by similar arguments as used in \eqref{eq0002}, we can see that $g_\delta(x,u)v\in L^{\frac65}(\bbr^3)$.  It follows from \cite[Lemma~2.1]{B02} that $\phi_{u,v,\delta}^0\in\D$ is the unique solution of the Poisson equation $-\Delta\phi=g_\delta(x,u)v$ in $\bbr^3$, which is given by
\begin{equation*}
\phi_{u,v,\delta}^0(x)=\frac{1}{4\pi}\int_{\bbr^3}\frac{g_\delta(y,u(y))v(y)}{|x-y|}dy.
\end{equation*}
By similar arguments as used in \eqref{eq0004}, we know that $|\int_{\bbr^3}G_\delta(x,u)\phi_{u,v,\delta}^0dx|<+\infty$.  Therefore, thanks to the Fubini theorem, we have $\Psi_\delta(u)\in C^1(\bbr^3, \bbr)$ and
\begin{equation*}
\langle \Psi_\delta'(u),v \rangle_{E_\lambda^*, E_\lambda}=2\int_{\bbr^3}\phi_{u,\delta}g_\delta(x,u)vdx
\end{equation*}
for every $u,v\in E_\lambda$ with $\lambda\geq1$, where $E_\lambda^*$ is the dual space of $E_\lambda$.  Now, by applying the condition $(F_1)$ and the H\"older and Sobolev inequalities in a standard way and applying a similar argument as used in \eqref{eq0024},
we can see that $\mathcal{E}_{\lambda,\delta}(u)$ is $C^1$ on $E_\lambda$ for $\lambda\geq1$.  Furthermore, for every $u,v\in E_\lambda$ with $\lambda\geq1$, we have
\begin{equation*}
\langle \mathcal{E}_{\lambda,\delta}'(u),v \rangle_{E_\lambda^*, E_\lambda}=\mathcal{D}_\lambda(u,v)+\int_{\bbr^3}\frac12\phi_{u,\delta}g_\delta(x,u)vdx-\int_{\bbr^3}f(x,u)vdx.
\end{equation*}
Hence, if $u\in E_\lambda$ is a critical point of $\mathcal{E}_{\lambda,\delta}(u)$, then $u$ is also a solution of $(\overline{\mathcal{S}}_{\lambda,\delta})$ due to the fact that $E_\lambda$ is embedded continuously into $\h$, that is, $\mathcal{E}_{\lambda,\delta}(u)$ is the corresponding functional of $(\overline{\mathcal{S}}_{\lambda,\delta})$ in $E_\lambda$ for $\lambda\geq1$.
\end{proof}

\subsection{Properties of the functional $\mathcal{D}_\lambda(u,u)$}
We have known from Lemma~\ref{lem0001} that $\mathcal{D}_\lambda(u,u)$ is well defined on $E_\lambda$ for $\lambda\geq1$.  By a similar argument as used in \eqref{eq0024}, we can see that $\mathcal{D}_\lambda(u,u)$ is actually $C^2$ on $E_\lambda$ for $\lambda\geq1$.  Furthermore, we also have the following lemma for $\mathcal{D}_\lambda(u,u)$, which is inspired by \cite{BT13} and \cite{ZLZ13}.
\begin{lemma}\label{lem0002}
There exists $\Lambda_0\geq1$ such that the augmented Morse index of $\mathcal{D}_\lambda(u,u)$ is less than or equal to $k_0$ uniformly for $u\in E_\lambda$ with $\lambda\geq\Lambda_0$, where $k_0=$dim$\mathcal{V}_b$ and $\mathcal{V}_b=\{u\in H_0^1(\Omega)\mid\int_{\Omega_a}(|\nabla u|^2+b(x)u^2)dx\leq0\}$.
\end{lemma}
\begin{proof}
If $b(x)\geq0$ on $\bbr^3$, then it is easy to see that $\mathcal{V}_b=\{0\}$ and $k_0=0$.  On the other hand, since $b(x)\geq0$ on $\bbr^3$, it is also easy to show that the augmented Morse index of $\mathcal{D}_\lambda(u,u)$ is equal to $0$ uniformly for $u\in E_\lambda$ with $\lambda\geq1$, which then implies that Lemma~\ref{lem0002} holds in this case.  It remains to show that Lemma~\ref{lem0002} is also true if $b(x)$ is sign-changing.  Indeed, since the condition $(B_3)$ holds, we can see that $b^{-1}([0, +\infty))\not=\emptyset$ in this case.  It follows from the condition $(B_1)$ that $\mathcal{F}_\lambda\not=\emptyset$ for every $\lambda\geq1$, where
\begin{equation*}
\mathcal{F}_\lambda=\{u\in E_\lambda\mid \text{supp}u\subset b^{-1}([0, +\infty))\}.
\end{equation*}
Let
\begin{equation*}
\mathcal{F}_\lambda^{\perp}=\{u\in E_\lambda\mid \langle u,v \rangle_\lambda=0\text{ for all } v\in \mathcal{F}_\lambda\}.
\end{equation*}
Then $E_\lambda=\mathcal{F}_\lambda\oplus\mathcal{F}_\lambda^{\perp}$.  Furthermore, by the condition $(B_1)$, we have $\mathcal{F}_\lambda^{\perp}\not=\{0\}$ for every $\lambda\geq1$ since $b(x)$ is sign-changing.  For the sake of clarity, the proof is further performed through the following several steps.

{\bf Step.~1}\quad We prove that for every $\lambda\geq1$, the operator $(-\Delta+\lambda a(x)+b^+(x))^{-1}b^-(x)$ has a sequence of positive eigenvalues $\{\alpha_j(\lambda)\}$ in $\mathcal{F}_\lambda^{\perp}$ satisfying $0<\alpha_1(\lambda)\leq\alpha_2(\lambda)\leq\cdots\leq\alpha_j(\lambda)\to+\infty$ as $j\to+\infty$, and the corresponding eigenfunctions are a basis of $\mathcal{F}_\lambda^{\perp}$.

Indeed, it is easy to see that the operator $(-\Delta+\lambda a(x)+b^+(x))^{-1}b^-(x)$ is linear and self-conjugate on $\mathcal{F}_\lambda^{\perp}$ for all $\lambda\geq1$.  Thanks to the condition $(B_3)$, $(-\Delta+\lambda a(x)+b^+(x))^{-1}b^-(x)$ is also compact on $\mathcal{F}_\lambda^{\perp}$ for all $\lambda\geq1$.  Thus, by \cite[Theorems~4.45 and~4.46]{W95}, the eigenvalue problem $-\Delta u+(\lambda a(x)+b^+(x))u=\alpha b^-(x)u$ has a sequence of positive eigenvalues $\{\alpha_j(\lambda)\}$ in $\mathcal{F}_\lambda^{\perp}$ satisfying $0<\alpha_1(\lambda)\leq\alpha_2(\lambda)\leq\cdots\leq\alpha_j(\lambda)\to+\infty$ as $j\to+\infty$.  Furthermore, $\{\alpha_j(\lambda)\}$ can be characterized by
\begin{equation}   \label{eq0005}
\alpha_j(\lambda)=\inf_{\text{dim}M\geq j, M\subset\mathcal{F}_\lambda^{\perp}}\sup\bigg\{\|u\|_\lambda^2\mid u\in M\text{ and }\int_{\bbr^3}b^-(x)u^2dx=1\bigg\},\quad j=1,2,3,\cdots,
\end{equation}
and the corresponding eigenfunctions $e_{j}(\lambda)$ can be chosen so that $\int_{\bbr^3}b^-(x)e_j^2(\lambda)dx=1$ for all $j\in\bbn$ and are a basis of $\mathcal{F}_\lambda^{\perp}$.

{\bf Step.~2}\quad We prove that $\alpha_j(\lambda)$ is nondecreasing for $\lambda$ and $\alpha_j(\lambda)\to \alpha_j^0$ as $\lambda\to+\infty$, where $\alpha_j^0$ is a positive eigenvalue of the following problem
\begin{equation}  \label{eq0006}
-\Delta u+b^+(x)u=\alpha b^-(x)u,\quad u\in H^1_0(\Omega_a).
\end{equation}

Indeed, let $\lambda_1\geq\lambda_2$, then by the definition of $E_\lambda$, we have $E_{\lambda_1}=E_{\lambda_2}$.  It follows that $\mathcal{F}_{\lambda_2}=\mathcal{F}_{\lambda_1}$, which implies $\mathcal{F}_{\lambda_2}^{\perp}=\mathcal{F}_{\lambda_1}^{\perp}$.  Note that $\|u\|_{\lambda_1}\geq\|u\|_{\lambda_2}$ for all $u\in E_{\lambda_1}$, so by the definition of $\alpha_j(\lambda_1)$ and $\alpha_j(\lambda_2)$, we can see that $\alpha_j(\lambda_2)\leq\alpha_j(\lambda_1)$, that is, $\alpha_j(\lambda)$ is nondecreasing for $\lambda$.  In what follows, we will show that $\alpha_j(\lambda)\to \alpha_j^0$ as $\lambda\to+\infty$.  By the conditions $(B_1)$ and $(B_3)$, for every $j\in\bbn$, there exists $\{\varphi_m\}_{1\leq m\leq j}\subset C_0^\infty(\bbr^3)$ such that supp$\varphi_m\subset\Omega_a\cap b^{-1}((-\infty, 0))$ and supp$\varphi_m\cap$ supp$\varphi_n=\emptyset$ for $m\not=n$.  Let $M_0=$span$\{\varphi_1, \cdots, \varphi_j\}$.  Then by \eqref{eq0005}, $\alpha_j(\lambda)\leq \alpha_j^*$, where
\begin{equation*}
\alpha_j^*=\sup\bigg\{\int_{\Omega_a}(|\nabla u|+b^+(x)u^2)dx\mid u\in M_0\text{ and }\int_{\Omega_a}b^-(x)u^2dx=1\bigg\}.
\end{equation*}
It follows from the choice of $e_j(\lambda)$ that $\|e_j(\lambda)\|_\lambda\leq\sqrt{\alpha_j^*}$.  Note that by \eqref{eq0023}, we have $\|e_j(\lambda)\|^2\leq(1+d_*)\|e_j(\lambda)\|_\lambda^2$, where $d_*>0$ is a constant given by \eqref{eq0024}.  So up to a subsequence, $e_j(\lambda)\rightharpoonup e_j$ weakly in $\h$ as $\lambda\to+\infty$.  By the condition $(A_1)$, $\|e_j(\lambda)\|_\lambda\leq\sqrt{\alpha_j^*}$ once more and the Fatou lemma, we can see that $\int_{\bbr^3}a(x)e_j(\lambda)^2dx\to0$ as $\lambda\to+\infty$ up to a subsequence and $\int_{\bbr^3}a(x)e_j^2dx=0$, which then together with the condition $(A_3)$, implies $e_j\in H_0^1(\Omega_a)$.  Since the condition $(A_2)$ holds, we have $|\mathcal{A}_\infty\cap (\bbr^3\backslash B_R)|\to0$ as $R\to+\infty$, where $B_R=\{x\in\bbr^3\mid|x|<R\}$.  It follows from the Sobolev embedding theorem, $\int_{\bbr^3}a(x)e_j(\lambda)^2dx\to0$ as $\lambda\to+\infty$ and $\int_{\bbr^3}a(x)e_j^2dx=0$ that $e_j(\lambda)\to e_j$ strongly in $L^2(\bbr^3)$ as $\lambda\to+\infty$ up to a subsequence.  Now, by the conditions $(A_2)$ and $(B_2)$, for every $\psi\in C_0^\infty(\Omega_a)$, we can see that
\begin{eqnarray*}
\int_{\Omega_a}\nabla e_j\nabla \psi+b^+(x)e_j\psi dx&=&\lim_{\lambda\to+\infty}\int_{\bbr^3}\nabla e_j(\lambda)\nabla \psi+b^+(x)e_j(\lambda)\psi dx\\
&=&\lim_{\lambda\to+\infty}\int_{\bbr^3}\nabla e_j(\lambda)\nabla \psi_\lambda+b^+(x)e_j(\lambda)\psi_\lambda dx\\
&=&\lim_{\lambda\to+\infty}\alpha_j(\lambda)\int_{\bbr^3}b^-(x)e_j(\lambda)\psi_\lambda dx\\
&=&\lim_{\lambda\to+\infty}\alpha_j(\lambda)\int_{\bbr^3}b^-(x)e_j(\lambda)\psi dx\\
&=&\alpha_j^0\int_{\Omega_a}b^-(x)e_j\psi dx,
\end{eqnarray*}
where $\psi_\lambda$ is the projection of $\psi$ in $\mathcal{F}_\lambda^{\perp}$.  Hence, $(e_j,\alpha_j^0)$ satisfies \eqref{eq0006}.  Note that $\alpha_j(\lambda)$ is nondecreasing for $\lambda$, so by Step.~1, we can see that $\alpha_j^0$ is positive.

Now, since $b(x)\in C(\bbr^3, \bbr)$ and $|\Omega_a|<+\infty$, by a similar argument as used in Step.~1, we can see that the eigenvalue problem \eqref{eq0005} has a sequence of positive eigenvalues $\{\overline{\alpha}_j\}$ satisfying $0<\overline{\alpha}_1\leq\overline{\alpha}_2\leq\cdots\leq\overline{\alpha}_j\to+\infty$ as $j\to+\infty$, and the corresponding eigenfunctions $\overline{e}_j$ are a basis of $H^1_0(\Omega_a)$.  Hence, $k_0=$dim$\mathcal{V}_b<+\infty$.  Suppose there exist $j\not=i$ such that $\alpha_j^0=\alpha_i^0=\overline{\alpha}_k$ for some $k\in\bbn$.  Then one of the following two cases must happen:
\begin{enumerate}
\item[$(1)$] $e_j=e_i$;
\item[$(2)$] $e_j\not=e_i$ and $\int_{\Omega_a}\nabla e_j\nabla e_i+b^+(x)e_je_idx=0$.
\end{enumerate}
If case~$(1)$ happen, then by Step.~1 and Step.~2, we have
\begin{eqnarray*}
2\overline{\alpha}_k=\lim_{\lambda\to+\infty}(\alpha_j(\lambda)+\alpha_i(\lambda))
=\lim_{\lambda\to+\infty}(\|e_j(\lambda)\|_\lambda^2+\|e_i(\lambda)\|_\lambda^2)
=\lim_{\lambda\to+\infty}(\|e_j(\lambda)-e_i(\lambda)\|_\lambda^2)=0.
\end{eqnarray*}
It is impossible.  Therefore, we must have the case~$(2)$.  Now, by Step.~2, we can see that there exists $\Lambda_0\geq1$ such that $\alpha_{k_0+1}(\lambda)>1$ for $\lambda\geq\Lambda_0$.  It follows from Step.~1 that $\alpha_j(\lambda)>1$ for all $j\geq k_0+1$ and $\lambda\geq\Lambda_0$, which implies that the augmented Morse index of $\mathcal{D}_\lambda(u,u)$ is less than or equal to $k_0$ uniformly for $u\in E_\lambda$ with $\lambda\geq\Lambda_0$.
\end{proof}

By Step.~1 of Lemma~\ref{lem0002}, $\mathcal{F}_{\lambda}^{\perp}=\mathcal{F}_{\lambda,-}^{\perp}\oplus\mathcal{F}_{\lambda,+}^{\perp}$ for every $\lambda\geq1$, where $\mathcal{F}_{\lambda,-}^{\perp}=$span$\{e_j(\lambda)\mid\alpha_j(\lambda)\leq1\}$ and $\mathcal{F}_{\lambda,+}^{\perp}=$span$\{e_j(\lambda)\mid\alpha_j(\lambda)>1\}$, which implies $E_\lambda=\mathcal{F}_{\lambda,-}^{\perp}\oplus\mathcal{F}_{\lambda,+}^{\perp}\oplus\mathcal{F}_{\lambda}$ for every $\lambda\geq1$.  Moreover, due to Lemma~\ref{lem0002} again, we know that dim$(\mathcal{F}_{\lambda,-}^{\perp})\leq k_0<+\infty$ for all $\lambda\geq\Lambda_0$ and $k_0$ is independent of $\lambda\geq\Lambda_0$.  Let $l_0=\max\{j\mid\overline{\alpha}_j<1\}$.  Then $l_0\leq k_0$ and $\overline{\alpha}_{l_0+1}>1$ due to the fact that $0\not\in\sigma(-\Delta+b(x), H_0^1(\Omega_a))$.  Now, we have the following important estimates for $\mathcal{D}_\lambda(u,u)$ on $E_\lambda$.
\begin{lemma}\label{lem0003}
Suppose $0\not\in\sigma_p(-\Delta+b(x), H^1_0(\Omega_a))$, then there exists $\Lambda_1\geq\Lambda_0$ such that for every $u\in E_\lambda$ with $\lambda\geq\Lambda_1$, $\mathcal{D}_\lambda(u,u)\leq-\frac12(1-\frac{1}{\overline{\alpha}_{l_0}})\|u\|_\lambda^2$ on $\mathcal{F}_{\lambda,-}^{\perp}$ and $\mathcal{D}_\lambda(u,u)\geq\frac12(1-\frac{1}{\overline{\alpha}_{l_0+1}})\|u\|_\lambda^2$ on $\mathcal{F}_{\lambda,+}^{\perp}\oplus\mathcal{F}_{\lambda}$.
\end{lemma}
\begin{proof}
By Step.~2 of Lemma~\ref{lem0002}, there exists $\Lambda_1\geq\Lambda_0$ such that $\frac{1}{\alpha_{l_0}(\lambda)}-1\geq\frac12(\frac{1}{\overline{\alpha}_{l_0}}-1)$ and $(1-\frac{1}{\alpha_{l_0+1}(\lambda)})\geq\frac12(1-\frac{1}{\overline{\alpha}_{l_0+1}})$ for all $\lambda\geq\Lambda_1$.  Now, by the definition of $\mathcal{F}_{\lambda,-}^{\perp}$ and Step.~1 of Lemma~\ref{lem0002}, it is easy to see that $\mathcal{D}_\lambda(u,u)\leq-\frac12(1-\frac{1}{\overline{\alpha}_{l_0}})\|u\|_\lambda^2$ on $\mathcal{F}_{\lambda,-}^{\perp}$ for all $u\in E_\lambda$ with $\lambda\geq\Lambda_1$.  On the other hand, it follows from Step.~1 of Lemma~\ref{lem0002} once more that $(1-\frac{1}{\alpha_{j}(\lambda)})\geq\frac12(1-\frac{1}{\overline{\alpha}_{l_0+1}})$ for all $\lambda\geq\Lambda_1$ and $j\geq l_0+1$.  Now, for every $u\in\mathcal{F}_{\lambda,+}^{\perp}\oplus\mathcal{F}_{\lambda}$, we respectively denote the projections of $u$ on $\mathcal{F}_{\lambda,+}^{\perp}$ and $\mathcal{F}_{\lambda}$ by $u_\lambda^{*}$ and $u_\lambda^{**}$.  Then we have
\begin{eqnarray*}
\mathcal{D}_\lambda(u,u)&=&\|u\|_\lambda^2-\int_{\bbr^3}b^-(x)u^2dx\\
&=&\|u_\lambda^{*}\|_\lambda^2+\|u_\lambda^{**}\|_\lambda^2-\int_{\bbr^3}b^-(x)(u_\lambda^{*})^2dx\\
&\geq&\|u_\lambda^{*}\|_\lambda^2+\|u_\lambda^{**}\|_\lambda^2-\frac{1}{\alpha_{l_0+1}(\lambda)}\|u_\lambda^{*}\|_\lambda^2\\
&\geq&\frac12(1-\frac{1}{\overline{\alpha}_{l_0+1}})\|u\|_\lambda^2,
\end{eqnarray*}
which completes the proof.
\end{proof}

\subsection{Critical points of the functional $\mathcal{E}_{\lambda,\delta}(u)$}
Since the condition $(B_3)$ holds, we can choose $v_0\in C_0^{\infty}(\bbr^3)$ such that supp$v_0\subset\Omega_a\cap b^{-1}([0, +\infty))$.  For every $R>0$, we define
\begin{equation*}
\mathcal{Q}_{\lambda,R}=\{u=z+tv_0\mid t\geq0,z\in\mathcal{F}_{\lambda,-}^{\perp}, \|u\|_\lambda\leq R\}.
\end{equation*}
Then we have the following lemma, which implies $\mathcal{E}_{\lambda,\delta}(u)$ has a linking structure in $E_\lambda$ for $\lambda\geq\Lambda_1$.
\begin{lemma}\label{lem0004}
There exist $f_1^*>0$, $\delta_0\in(0, \min\{1,\frac12(S^{-1}\|K\|_{L^2(\bbr^3)}^{-2}f_2)^{\frac12}\})$, $\rho_0>0$ and $R_0>0$ which are all independent on $\lambda\geq\Lambda_1$ such that
\begin{equation*}
\inf_{\mathcal{S}_{\lambda,\rho_0}\cap(\mathcal{F}_{\lambda,+}^{\perp}\oplus\mathcal{F}_{\lambda})}\mathcal{E}_{\lambda,\delta_0}(u)
>\sup_{\partial\mathcal{Q}_{\lambda,R_0}}\mathcal{E}_{\lambda,\delta_0}(u)
\end{equation*}
for $f_1<f_1^*$ and $\lambda\geq\Lambda_1$, where $f_1$ is given by the condition $(F_1)$ and $\mathcal{S}_{\lambda,\rho_0}=\{u\in E_\lambda\mid\|u\|_\lambda=\rho_0\}$.
\end{lemma}
\begin{proof}
By the conditions $(F_1)$-$(F_2)$, there exist $\ve,C_\ve>0$ such that $|f(x,t)|\leq\ve|t|+C_\ve|t|^{p-1}$.  It follows that $|\int_{\bbr^3}F(x,u)dx|\leq\ve\|u\|_{L^2(\bbr^3)}^2+C_\ve\|u\|_{L^p(\bbr^3)}^p$.  Since $\Lambda_1\geq1$, by \eqref{eq0023} and the Ho\"lder and Sobolev inequalities, we have
\begin{equation}  \label{eq0027}
\|u\|_{L^p(\bbr^3)}^p\leq d_*^{\frac{6-p}{4}}S^{-\frac{3p-6}{4}}\|u\|_\lambda^p,\quad\text{for all }u\in E_\lambda\text{ with }\lambda\geq\Lambda_1,
\end{equation}
where $d_*>0$ is a constant given by \eqref{eq0024}.  Let $\ve_0=\frac{1}{8 d_*}(1-\frac{1}{\overline{\alpha}_{l_0+1}})$.  Note that $\phi_{u,\delta}\geq0$ on $\bbr^3$ by the construction of $g_\delta(x,t)$ and the maximum principle, so by Lemma~\ref{lem0003} and \eqref{eq0023} once more, we can see that
\begin{eqnarray*}
\mathcal{E}_{\lambda,\delta}(u)&=&\frac12\mathcal{D}_\lambda(u,u)+\frac14\int_{\bbr^3}\phi_{u,\delta}G_\delta(x,u)dx-\int_{\bbr^3}F(x,u)dx\\
&\geq&\frac12\mathcal{D}_\lambda(u,u)-|\int_{\bbr^3}F(x,u)dx|\\
&\geq&\frac18(1-\frac{1}{\overline{\alpha}_{l_0+1}})\|u\|_\lambda^2-C_*\|u\|_\lambda^p
\end{eqnarray*}
for all $u\in \mathcal{F}_{\lambda,+}^{\perp}\oplus\mathcal{F}_{\lambda}$ with $\lambda\geq\Lambda_1$, where $C_*=C_{\ve_0}d_*^{\frac{6-p}{4}}S^{-\frac{3p-6}{4}}$.  Let $\rho_0=\bigg(\frac{2d_*\ve_0}{pC_*}\bigg)^{\frac{1}{p-2}}$.  Then
\begin{equation}  \label{eq0008}
\inf_{\mathcal{S}_{\lambda,\rho_0}\cap(\mathcal{F}_{\lambda,+}^{\perp}\oplus\mathcal{F}_{\lambda})}\mathcal{E}_{\lambda,\delta}(u)
\geq\frac{p-2}{p}\bigg(\frac{2d_*\ve_0}{pC_*}\bigg)^{\frac{2}{p-2}}\quad\text{for }\lambda\geq\Lambda_1.
\end{equation}
In what follows, we will prove that there exist $\delta_0\in(0, 1)$ and $R_0>\rho_0$ independent of $\lambda\geq\Lambda_1$ such that $\sup_{\partial\mathcal{Q}_{\lambda,R_0}}\mathcal{E}_{\lambda,\delta_0}(u)<\frac{p-2}{p}\bigg(\frac{2d_*\ve_0}{pC_*}\bigg)^{\frac{2}{p-2}}$ for all $\lambda\geq\Lambda_1$.  For the sake of clarity, the proof will be further performed through the following two Claims.

{\bf Claim~1}\quad There exists $R_0>\rho_0$ independent of $\lambda\geq\Lambda_1$ such that
\begin{equation*}
\sup_{\{u=z+tv_0\mid t\geq0,z\in\mathcal{F}_{\lambda,-}^{\perp}, \|u\|_\lambda=R\}}\mathcal{E}_{\lambda,\delta}(u)<0
\end{equation*}
for all $\delta\in(0 ,1)$ and $\lambda\geq\Lambda_1$.

Indeed, since the condition $(K_1)$ holds and $E_\lambda$ with $\lambda\geq\Lambda_1$ is embedded continuously into $\h$, for every $u\in E_\lambda$ with $\lambda\geq\Lambda_1$, we have $K(x)|u(x)|\in L^{\frac65}(\bbr^3)$ due to the H\"older inequality.  It follows from \cite[Lemma~2.1]{B02} that the Poisson equation $-\Delta\phi=K(x)|u(x)|$ in $\bbr^3$ has a unique solution in $D^{1,2}(\bbr^3)$ which can be given by
\begin{equation*}
\overline{\phi}_{u}(x)=\frac{1}{4\pi}\int_{\bbr^3}\frac{K(y)|u(y)|}{|x-y|}dy.
\end{equation*}
Moreover, by the H\"older and Sobolev inequalities, $\|\overline{\phi}_{u}\|_{L^6(\bbr^3)}\leq S^{-1}\|K\|_{L^2(\bbr^3)}\|u\|_{L^{3}(\bbr^3)}$.  By the construction of $G_\delta(x,u)$ and the condition $(K_2)$, we can see that
\begin{eqnarray}
\int_{\bbr^3}\phi_{u,\delta}G_\delta(x,u)dx&\leq&\int_{\bbr^3\backslash\Omega_K}2\delta K(x)|u(x)|\int_{\bbr^3\backslash\Omega_K}\frac{2\delta K(y)|u(y)|}{4\pi|x-y|}dydx\notag\\
&\leq&\int_{\bbr^3}4\delta^2 K(x)|u(x)|\overline{\phi}_{u}(x)dx\notag\\
&\leq&4\delta^2S^{-1}\|K\|_{L^2(\bbr^3)}^2\|u\|_{L^3(\bbr^3)}^2,\label{eq0007}
\end{eqnarray}
which implies
\begin{equation}\label{eq0026}
\sup_{\{u=z+tv_0\mid t\geq0,z\in\mathcal{F}_{\lambda,-}^{\perp}, \|u\|_\lambda=R\}}\mathcal{E}_{\lambda,\delta}(u)\leq\sup_{\{u=z+tv_0\mid t\geq0,z\in\mathcal{F}_{\lambda,-}^{\perp}, \|u\|_\lambda=R\}}J(u)\quad\text{for }0<\delta<1,
\end{equation}
where $J(u)=\frac12\|u\|_\lambda^2+S^{-1}\|K\|_{L^2(\bbr^3)}^2\|u\|_{L^3(\bbr^3)}^2-\int_{\bbr^3}F(x,u)dx$.  Since $\lambda\geq\Lambda_1$, we have dim$(\mathcal{F}_{\lambda,-}^{\perp}\oplus\mathbb{R}v_0)\leq k_0+1$ by Lemma~\ref{lem0002} for some $k_0$ independent of $\lambda\geq\Lambda_1$.  Hence, there exists $d_{k_0},d_{k_0}'>0$ independent on $\lambda\geq\Lambda_1$ such that $\|u\|_\lambda\leq d_{k_0}\|u\|_{L^3(\bbr^3)}\leq d_{k_0}'\|u\|_{L^{2}(\bbr^3)}$ for all $u\in\mathcal{F}_{\lambda,-}^{\perp}\oplus\mathbb{R}v_0$.  By the condition $(F_3)$ and the Fatou lemma, for every $u\in\mathcal{F}_{\lambda,-}^{\perp}\oplus\mathbb{R}v_0$ with $\|u\|_\lambda=1$, we have
\begin{equation*}
\lim_{t\to+\infty}\frac{J(tu)}{t^3}\leq\lim_{t\to+\infty}\frac{d_{k_0}'d_*+d_{k_0}'d_{k_0}^{-1}d_*}{2t}-\lim_{t\to+\infty}\frac12\int_{\bbr^3}\frac{2F(x,tu)}{|tu|^3}u^3dx
=-\infty,
\end{equation*}
which together with \eqref{eq0026}, implies there exists $R_0>0$ independent on $\lambda\geq\Lambda_1$ such that
\begin{equation*}
\sup_{\{u=z+tv_0\mid t\geq0,z\in\mathcal{F}_{\lambda,-}^{\perp}, \|u\|_\lambda=R_0\}}\mathcal{E}_{\lambda,\delta}(u)<0
\end{equation*}
for all $\delta\in(0 ,1)$ and $\lambda\geq\Lambda_1$.

{\bf Claim~2}\quad There exists $\delta_0\in(0, \min\{1,\frac12(S^{-1}\|K\|_{L^2(\bbr^3)}^{-2}f_2)^{\frac12}\})$ and $f_1^*>0$ independent of $\lambda\geq\Lambda_1$ such that
\begin{equation*}
\sup_{\{u\in\mathcal{F}_{\lambda,-}^{\perp}\mid \|u\|_\lambda\leq R_0\}}\mathcal{E}_{\lambda,\delta_0}(u)\leq\frac{p-2}{2p}\bigg(\frac{2d_*\ve_0}{pC_*}\bigg)^{\frac{2}{p-2}}
\end{equation*}
for $f_1<f_1^*$ and $\lambda\geq\Lambda_1$.

Indeed, by Lemma~\ref{lem0003} and \eqref{eq0007}, we have
\begin{equation*}
\sup_{\{u\in\mathcal{F}_{\lambda,-}^{\perp}\mid \|u\|_\lambda\leq R_0\}}\mathcal{E}_{\lambda,\delta}(u)\leq\sup_{\{u\in\mathcal{F}_{\lambda,-}^{\perp}\mid \|u\|_\lambda\leq R_0\}}(\delta^2S^{-1}\|K\|_{L^2(\bbr^3)}^2\|u\|_{L^3(\bbr^3)}^2-\int_{\bbr^3}F(x,u)dx).
\end{equation*}
It follows from the condition $(F_1)$, \eqref{eq0023} and \eqref{eq0027} that
\begin{eqnarray*}
&&\sup_{\{u\in\mathcal{F}_{\lambda,-}^{\perp}\mid \|u\|_\lambda\leq R_0\}}(\delta^2S^{-1}\|K\|_{L^2(\bbr^3)}^2\|u\|_{L^3(\bbr^3)}^2-\int_{\bbr^3}F(x,u)dx)\\
&\leq&\delta^2S^{-1}\|K\|_{L^2(\bbr^3)}^2\bigg(\frac{d_*}{S}\bigg)^{\frac12}R_0^2+f_1(d_*R_0^2+d_*^{\frac{6-p}{4}}S^{-\frac{3p-6}{4}}R_0^p).
\end{eqnarray*}
Therefore, there exists $\delta_0\in(0, 1)$ and $f_1^*>0$ independent of $\lambda\geq\Lambda_1$ such that
\begin{equation*}
\sup_{\{u\in\mathcal{F}_{\lambda,-}^{\perp}\mid \|u\|_\lambda\leq R_0\}}\mathcal{E}_{\lambda,\delta_0}(u)\leq\frac{p-2}{2p}\bigg(\frac{2d_*\ve_0}{pC_*}\bigg)^{\frac{2}{p-2}}
\end{equation*}
for $f_1<f_1^*$ and $\lambda\geq\Lambda_1$.

Now, the conclusion follows immediately from \eqref{eq0008}, Claim~1 and Claim~2.
\end{proof}

Since Lemma~\ref{lem0004} holds, by the well known linking theorem (cf. \cite{R86}), $\mathcal{E}_{\lambda,\delta_0}(u)$ has a  $(C)_{c_\lambda}$ sequence in $E_\lambda$ with $\lambda\geq\Lambda_1$.  That is, there exists $\{u_{\lambda,n}\}\subset E_\lambda$ such that $\mathcal{E}_{\lambda,\delta_0}(u_{\lambda,n})=c_\lambda+o_n(1)$ and $(1+\|u_{\lambda,n}\|_{\lambda})\mathcal{E}_{\lambda,\delta_0}'(u_{\lambda,n})=o_n(1)$ strongly in $E_\lambda^*$.  Furthermore, we have $c_\lambda\in[\frac{p-2}{p}\bigg(\frac{2d_*\ve_0}{pC_*}\bigg)^{\frac{2}{p-2}}, \frac12R_0^2+\frac{p-2}{p}\bigg(\frac{2d_*\ve_0}{pC_*}\bigg)^{\frac{2}{p-2}}+f_1(d_*R_0^2+d_*^{\frac{6-p}{4}}S^{-\frac{3p-6}{4}}R_0^p)]$.

\begin{lemma}\label{lem0005}
There exist $\Lambda_2\geq\Lambda_1$ and $C^*>0$ independent of $\lambda\geq\Lambda_2$ such that $\|u_{\lambda,n}\|_\lambda\leq C^*+o_n(1)$ for all $\lambda\geq\Lambda_2$.
\end{lemma}
\begin{proof}
By the construction of $g_\delta(x,t)$ and the conditions $(K_1)$ and $(F_4)$, we can see that
\begin{eqnarray}
c_\lambda+o_n(1)&=&\mathcal{E}_{\lambda,\delta_0}(u_{\lambda,n})-\frac12\langle\mathcal{E}_{\lambda,\delta_0}'(u_{\lambda,n}), u_{\lambda,n}\rangle_{E_\lambda^*,E_\lambda}\notag\\
&=&\frac14\int_{\bbr^3}\phi_{u_{\lambda,n},\delta}(G_\delta(x,u_{\lambda,n})-g_\delta(x,u_{\lambda,n})u_{\lambda,n})dx\notag\\
&&+\frac12\int_{\bbr^3}f(x,u_{\lambda,n})u_{\lambda,n}-2F(x,u_{\lambda,n})dx\notag\\
&\geq&-2\delta_0^2\int_{\bbr^3\backslash\Omega_k}\int_{\bbr^3\backslash\Omega_k}\frac{K(y)K(x)|u_{\lambda,n}(y)||u_{\lambda,n}(x)|}{4\pi|x-y|}dydx
+\frac{f_2}{2}\|u_{\lambda,n}\|_{L^3(\bbr^3)}^3\notag\\
&\geq&-2\delta_0^2\int_{\bbr^3}\int_{\bbr^3}\frac{K(y)K(x)|u_{\lambda,n}(y)||u_{\lambda,n}(x)|}{4\pi|x-y|}dydx
+\frac{f_2}{2}\|u_{\lambda,n}\|_{L^3(\bbr^3)}^3.\label{eq0010}
\end{eqnarray}
Since $K(x)|u_{\lambda,n}(x)|\in L^{\frac65}(\bbr^3)$ for every $n\in\bbn$, by the condition $(K_1)$ once more, we can follow the argument as used in \eqref{eq0007} to obtain that
\begin{equation}   \label{eq0009}
\bigg|\int_{\bbr^3\backslash\Omega_k}\int_{\bbr^3\backslash\Omega_k}\frac{K(y)K(x)|u_{\lambda,n}(y)||u_{\lambda,n}(x)|}{4\pi|x-y|}dydx\bigg|
\leq S^{-1}\|K\|_{L^2(\bbr^3)}^2\|u_{\lambda,n}\|_{L^3(\bbr^3)}^2.
\end{equation}
By \eqref{eq0010} and \eqref{eq0009}, we have $c_\lambda+o_n(1)+2\delta_0^2S^{-1}\|K\|_{L^2(\bbr^3)}^2\|u_{\lambda,n}\|_{L^3(\bbr^3)}^2\geq f_2\|u_{\lambda,n}\|_{L^3(\bbr^3)}^3$.  It follows from the Young inequality, the choice of $\delta_0$ and the fact $c_\lambda\leq \frac12R_0^2+\frac{p-2}{p}\bigg(\frac{2d_*\ve_0}{pC_*}\bigg)^{\frac{2}{p-2}}+f_1(d_*R_0^2+d_*^{\frac{6-p}{4}}S^{-\frac{3p-6}{4}}R_0^p)$ that $\|u_{\lambda,n}\|_{L^3(\bbr^3)}\leq C_0+o_n(1)$ for some $C_0>0$ independent of $\lambda\geq\Lambda_1$.  Now, thanks to the conditions $(A_2)$ and $(B_2)$-$(B_3)$ and the H\"older inequality, we have
\begin{eqnarray}
\int_{\bbr^3}b^-(x)|u_{\lambda,n}|^2dx&\leq&\int_{\mathcal{B}_0}b_0(1+\sqrt{a(x)})|u_{\lambda,n}|^2dx\notag\\
&\leq&b_0(1+\sqrt{a_\infty})|\mathcal{B}_0|^{\frac13}\|u_{\lambda,n}\|_{L^3(\bbr^3)}^2
+\frac{1}{\lambda\sqrt{a_\infty}}\|u_{\lambda,n}\|_\lambda^2\notag\\
&\leq&b_0(1+\sqrt{a_\infty})|\mathcal{B}_0|^{\frac13}C_0^2+\frac{1}{\lambda\sqrt{a_\infty}}\|u_{\lambda,n}\|_\lambda^2+o_n(1),\label{eq0028}
\end{eqnarray}
where $\mathcal{B}_0$ is given by the condition $(B_3)$.
By the condition $(F_5)$, \eqref{eq0009} and \eqref{eq0028}, we can obtain that
\begin{eqnarray*}
c_\lambda+o_n(1)&=&\mathcal{E}_{\lambda,\delta_0}(u_{\lambda,n})-\frac1\mu\langle\mathcal{E}_{\lambda,\delta_0}'(u_{\lambda,n}), u_{\lambda,n}\rangle_{E_\lambda^*, E_\lambda}\\
&=&(\frac12-\frac1\mu)\mathcal{D}_\lambda(u_{\lambda,n},u_{\lambda,n})+\int_{\bbr^3}\phi_{u_{\lambda,n},\delta}(\frac14 G_\delta(x,u_{\lambda,n})-\frac1{2\mu}g_\delta(x,u_{\lambda,n})u_{\lambda,n})dx\\
&&+\frac1\mu\int_{\bbr^3}f(x,u_{\lambda,n})u_{\lambda,n}-\mu F(x,u_{\lambda,n})dx\\
&\geq&(\frac12-\frac1\mu-\frac{1}{\lambda\sqrt{a_\infty}})\|u_{\lambda,n}\|_\lambda^2-b_0(1+\sqrt{a_\infty})|\mathcal{B}_0|^{\frac13}C_0^2
-f_3\|u_{\lambda,n}\|_{L^3(\bbr^3)}^3\notag\\
&&+\frac14\int_{\bbr^3}\phi_{u_{\lambda,n},\delta}(G_\delta(x,u_{\lambda,n})-g_\delta(x,u_{\lambda,n})u_{\lambda,n})dx\notag\\
&\geq&(\frac12-\frac1\mu-\frac{1}{\lambda\sqrt{a_\infty}})\|u_{\lambda,n}\|_\lambda^2-b_0(1+\sqrt{a_\infty})|\mathcal{B}_0|^{\frac13}C_0^2
-f_3\|u_{\lambda,n}\|_{L^3(\bbr^3)}^3\notag\\
&&-S^{-1}\|K\|_{L^2(\bbr^3)}^2\|u_{\lambda,n}\|_{L^3(\bbr^3)}^2.
\end{eqnarray*}
Therefore, by $c_\lambda\leq \frac12R_0^2+\frac{p-2}{p}\bigg(\frac{2d_*\ve_0}{pC_*}\bigg)^{\frac{2}{p-2}}+f_1(d_*R_0^2+d_*^{\frac{6-p}{4}}S^{-\frac{3p-6}{4}}R_0^p)$, there exists $\Lambda_2\geq\Lambda_1$ such that $\|u_{\lambda,n}\|_\lambda\leq C^*+o_n(1)$ for $\lambda\geq\Lambda_2$, where $C^*$ is independent of $\lambda\geq\Lambda_2$.
\end{proof}

We close this section by the following.
\begin{proposition}\label{prop0001}
There exists $\Lambda_3\geq\Lambda_2$ such that $(\overline{\mathcal{S}}_{\lambda,\delta_0})$ has a nontrivial solution $u_\lambda$ for $\lambda\geq\Lambda_3$.
\end{proposition}
\begin{proof}
By Lemma~\ref{lem0005}, $\{u_{\lambda,n}\}$ is bounded in $E_\lambda$ for $\lambda\geq\Lambda_2$.  It follows that $u_{\lambda,n}\rightharpoonup u_\lambda$ weakly in $E_\lambda$ for some $u_\lambda\in E_\lambda$ as $n\to\infty$.  Suppose $u_\lambda=0$.  Then by the condition $(K_1)$ and a similar argument as used in \cite[Lemma~2.1]{ZLZ13}, we can see that $\phi_{u_{\lambda,n},\delta_0}\to0$ strongly in $D^{1,2}(\bbr^3)$ as $n\to\infty$.  Since $\{u_{\lambda,n}\}$ is bounded in $E_\lambda$ for $\lambda\geq\Lambda_2$, by the condition $(K_1)$ once more and the construction of $g_{\delta_0}(x,t)$, we have
\begin{equation}\label{eq0029}
\int_{\bbr^3}\phi_{u_{\lambda,n},\delta_0}G_{\delta_0}(x,u_{\lambda,n})dx=\int_{\bbr^3}\phi_{u_{\lambda,n},\delta_0}g_{\delta_0}(x,u_{\lambda,n})u_{\lambda,n}dx
=o_n(1).
\end{equation}
On the other hand, since the condition $(B_3)$ holds, $|\mathcal{B}_0\cap (\bbr^3\backslash B_R)|\to0$ as $R\to+\infty$.  Therefore, by the Sobolev embedding theorem and the fact that $E_\lambda$ is embedded continuously into $\h$ for $\lambda\geq\Lambda_2$, we can obtain
$\int_{\bbr^3}b^-(x)u_{\lambda,n}^2dx=o_n(1)$, which together with \eqref{eq0029} and the fact that $\{u_{\lambda,n}\}$ is a $(C)_{c_\lambda}$ sequence of $\mathcal{E}_{\lambda,\delta_0}(u)$, implies
\begin{equation}\label{eq0011}
c_\lambda+o_n(1)=\frac12\|u_{\lambda,n}\|_\lambda^2-\int_{\bbr^3}F(x,u_{\lambda,n})dx
\end{equation}
and
\begin{equation}\label{eq0012}
o_n(1)=\|u_{\lambda,n}\|_\lambda^2-\int_{\bbr^3}f(x,u_{\lambda,n})u_{\lambda,n}dx.
\end{equation}
By the conditions $(F_1)$-$(F_2)$ and \eqref{eq0023} and \eqref{eq0012}, for every $\ve>0$, we have
\begin{eqnarray}
\|u_{\lambda,n}\|_\lambda^2&=&\int_{\bbr^3}f(x,u_{\lambda,n})u_{\lambda,n}dx+o_n(1)\notag\\
&\leq&\ve d_*\|u_{\lambda,n}\|_\lambda^2+C_\ve \|u_{\lambda,n}\|_{L^p(\bbr^3)}^p+o_n(1),\label{eq0013}
\end{eqnarray}
where $d_*>0$ is a constant given by \eqref{eq0024} and $C_\ve>0$ is a constant independent of $\lambda\geq\Lambda_2$.  Since the condition $(A_2)$ holds, $|\mathcal{A}_\infty\cap(\bbr^3\backslash B_R)|\to0$ as $R\to+\infty$, which together with the Sobolev embedding theorem, $u_{\lambda,n}\rightharpoonup 0$ weakly in $E_\lambda$ as $n\to\infty$ and the fact that $E_\lambda$ is embedded continuously into $\h$ for $\lambda\geq\Lambda_2$, implies $\int_{\mathcal{A}_\infty}u_{\lambda,n}^2dx=o_n(1)$.  Now, by the condition $(A_2)$ and the H\"older and Sobolev inequalities, we can see that
\begin{eqnarray}
\int_{\bbr^3}|u_{\lambda,n}|^pdx&\leq&\bigg(\int_{\bbr^3}|u_{\lambda,n}|^2dx\bigg)^{\frac{6-p}{4}}\bigg(\int_{\bbr^3}|u_{\lambda,n}|^6dx\bigg)^{\frac{3p-6}{12}}
\notag\\
&\leq&\bigg(\frac{1}{\lambda a_\infty}\int_{\bbr^3}\lambda a(x)|u_{\lambda,n}|^2dx+o_n(1)\bigg)^{\frac{6-p}{4}}S^{-\frac{3p-6}{4}}\|u_{\lambda,n}\|_{\lambda}^{\frac{3p-6}{2}}\notag\\
&\leq&\bigg(\frac{1}{\lambda a_\infty}\bigg)^{\frac{6-p}{2}}S^{-\frac{3p-6}{4}}\|u_{\lambda,n}\|_{\lambda}^p+o_n(1)\|u_{\lambda,n}\|_{\lambda}^{\frac{3p-6}{2}}.\label{eq0030}
\end{eqnarray}
Thanks to Lemma~\ref{lem0005}, \eqref{eq0013} and \eqref{eq0030}, we obtain
\begin{equation*}
\bigg(1-\ve d_*-C_\ve\bigg(\frac{1}{\lambda a_\infty}\bigg)^{\frac{6-p}{2}}S^{-\frac{3p-6}{4}}(C^*)^{p-2}\bigg)\|u_{\lambda,n}\|_{\lambda}^2\leq o_n(1).
\end{equation*}
Take $\ve=\frac{1}{2d_*}$.  Then there exists $\Lambda_3\geq\Lambda_2$ such that $\|u_{\lambda,n}\|_{\lambda}^2=o_n(1)$ for $\lambda\geq\Lambda_3$.
Hence, $u_{\lambda,n}\to0$ strongly in $E_\lambda$ with $\lambda\geq\Lambda_3$ as $n\to\infty$.  Since $E_\lambda$ is embedded continuously into $\h$ for $\lambda\geq\Lambda_3$, we have $\int_{\bbr^3}F(x,u_{\lambda,n})dx=o_n(1)$ by the condition $(F_1)$ and the H\"older and Sobolev inequalities, which together with \eqref{eq0011}, implies $c_\lambda=0$ for $\lambda\geq\Lambda_3$.  It is impossible since $c_\lambda\geq\frac{p-2}{p}\bigg(\frac{2d_*\ve_0}{pC_*}\bigg)^{\frac{2}{p-2}}>0$ for $\lambda\geq\Lambda_3$.  We close the proof by showing that $u_\lambda$ is also a solution of $(\overline{\mathcal{S}}_{\lambda,\delta_0})$ for $\lambda\geq\Lambda_2$.  Indeed, since $u_{\lambda,n}\rightharpoonup u_\lambda$ weakly in $E_\lambda$ as $n\to\infty$, by the condition $(K_1)$ and a similar argument as used in \cite[Lemma~2.1]{ZLZ13}, we can see that $\phi_{u_{\lambda,n},\delta_0}\to\phi_{u_{\lambda},\delta_0}$ strongly in $D^{1,2}(\bbr^3)$ as $n\to\infty$.  Now, since the condition $(B_3)$ and $(F_1)$ hold and $\{u_{\lambda,n}\}$ is a $(C)_{c_\lambda}$ sequence of $\mathcal{E}_{\lambda,\delta_0}(u)$, by a standard argument, we can show that $\mathcal{E}_{\lambda,\delta_0}'(u_\lambda)=0$ for $\lambda\geq\Lambda_3$.  It follows from Lemma~\ref{lem0001} that $u_\lambda$ is a nontrivial solution of $(\overline{\mathcal{S}}_{\lambda,\delta_0})$ for $\lambda\geq\Lambda_3$.
\end{proof}

\section{Proof of Theorem~\ref{thm0001}}
For the sake of convenience, we also assume the conditions $(A_1)$--$(A_4)$, $(B_1)$--$(B_3)$, $(K_1)$--$(K_2)$ and $(F_1)$--$(F_5)$ hold with $a(x)\geq0$ in this section as in section~2.  By Proposition~\ref{prop0001}, $(\overline{\mathcal{S}}_{\lambda,\delta_0})$ has a nontrivial solution $u_\lambda$ for $\lambda\geq\Lambda_3$, where $\Lambda_3$ is given by Proposition~\ref{prop0001}.  In this section, we will verify that $u_\lambda$ is also the solution of $(\mathcal{SP}_{\lambda})$ for $\lambda$ sufficiently large.  By the choice of $\Omega_a'$, we can find $\Omega_a\subset\Omega_a''\subset\Omega_a'$ such that dist$(\Omega_a,\bbr^3\backslash\Omega_a'')>0$ and dist$(\Omega_a'', \bbr^3\backslash\Omega_a')>0$.
\begin{lemma}\label{lem0006}
We have $\int_{\bbr^3\backslash\Omega_a''}u_\lambda^2dx\to0$ as $\lambda\to+\infty$.
\end{lemma}
\begin{proof}
The idea of this proof comes from \cite{ZLZ13} and it was also used in \cite{SW141}.  Suppose the contrary, there exists $\{\lambda_n\}$ satisfying $\lambda_n\to+\infty$ as $n\to\infty$ and $\gamma_0>0$ such that $\int_{\bbr^3\backslash\Omega_a''}u_{\lambda_n}^2dx\geq\gamma_0$.  Without loss of generality, we assume $\lambda_n\geq\Lambda_3$ for all $n\in\bbn$.
By Lemma~\ref{lem0005}, we have $\|u_{\lambda_n}\|_{\lambda_n}\leq C^*$, where $C^*>0$ is independent of $n\in\bbn$.  By \eqref{eq0023}, $\|u_{\lambda_n}\|\leq(1+d_*)C^*$, where $d_*>0$ is a constant given by \eqref{eq0024}.  Therefore, $u_{\lambda_n}\rightharpoonup u_0$ in $\h$ for some $u_0\in\h$ as $n\to\infty$.  Thanks to the condition $(A_1)$, $\|u_{\lambda_n}\|_{\lambda_n}\leq C^*$ and the Fatou lemma, we can see that $\int_{\bbr^3}a(x)u_{\lambda_n}^2dx=o_n(1)$ and $\int_{\bbr^3}a(x)u_{0}^2dx=0$.  It follows from the condition $(A_3)$ that $u_0\in H^1_0(\Omega_a)$ with $u_0\equiv0$ outside $\Omega_a$.  Therefore, by the choice of $\Omega_a''$, we have
\begin{equation*}
\int_{\bbr^3}|u_{\lambda_n}-u_0|^2dx=\int_{\Omega_a''}|u_{\lambda_n}-u_0|^2dx+\int_{\bbr^3\backslash\Omega_a''}|u_{\lambda_n}|^2dx
\geq\int_{\bbr^3\backslash\Omega_a''}|u_{\lambda_n}|^2dx\geq\gamma_0,
\end{equation*}
which together with the Lions lemma \cite{L84}, implies that there exist $\{x_n\}\subset\bbr^3$ satisfying $|x_n|\to+\infty$ as $n\to\infty$ and $r_0,\sigma_0>0$ such that
\begin{equation*}
\int_{B_{r_0}(x_n)}|u_{\lambda_n}-u_0|^2dx\geq\sigma_0,
\end{equation*}
where $B_{r_0}(x_n)=\{x\in\bbr^3\mid|x-x_n|\leq r_0\}$.  Since the condition $(A_2)$ holds and $|x_n|\to+\infty$ as $n\to\infty$, it is easy to see that $|B_{r_0}(x_n)\cap\mathcal{A}_\infty|\to0$ as $n\to\infty$.  Now, by the condition $(A_2)$, $\|u_{\lambda_n}\|_{\lambda_n}\leq C^*$ and the H\"older and Sobolev inequalities once more, we can obtain that
\begin{eqnarray*}
C^*&\geq&\int_{\bbr^3}|\nabla u_{\lambda_n}|^2+(\lambda_na(x)+b^+(x))|u_{\lambda_n}|^2dx\\
&\geq&\lambda_n\int_{B_{r_0}(x_n)\cap(\bbr^3\backslash\mathcal{A}_\infty)}a(x)|u_{\lambda_n}|^2dx\\
&\geq&\lambda_na_\infty\int_{B_{r_0}(x_n)\cap(\bbr^3\backslash\mathcal{A}_\infty)}|u_{\lambda_n}-u_0|^2dx\\
&\geq&\lambda_na_\infty(\int_{B_{r_0}(x_n)}|u_{\lambda_n}-u_0|^2dx-4\int_{B_{r_0}(x_n)}|u_0|^2dx
-4\int_{B_{r_0}(x_n)\cap\mathcal{A}_\infty}|u_{\lambda_n}|^2dx)\\
&\geq&\lambda_na_\infty(\sigma_0+o_n(1))\to+\infty,
\end{eqnarray*}
which is a contradiction.
\end{proof}

With Lemma~\ref{lem0006} in hands, we can obtain the following.
\begin{lemma}\label{lem0007}
If $3<p\leq4$, then there exists $\Lambda_*\geq\Lambda_3$ such that $|u_\lambda(x)|\leq\delta_0$ a.e. on $\bbr^3\backslash\Omega_a'$ for $\lambda\geq\Lambda_*$.
\end{lemma}
\begin{proof}
Let $r_*=\min\{\frac13$dist$(\Omega_a'',\bbr^3\backslash\Omega_a'),(f_1+b_0(1+\sqrt{a_\infty}))^{-\frac12},f_1^{-\frac12}\}$.  Then for every $y\in\bbr^3\backslash\Omega_a'$, $B_{2r_*}(y)\subset\bbr^3\backslash\Omega_a''$, where $B_{2r_*}(y)=\{x\in\bbr^3\mid|y-x|\leq2r_*\}$.  Let $\rho(x)\in C_0^\infty(\bbr^3, [0, 1])$ be given by
\begin{equation*}
\rho(x)=\left\{\aligned&1,\quad&x\in B_{r_2}(y),\\
&0,\quad&x\in\bbr^3\backslash B_{r_1}(y),\endaligned\right.
\end{equation*}
where $r_*<r_2<r_1<2r_*$.  Moreover, $|\nabla\rho(x)|\leq\frac{C_4}{r_1-r_2}$ and $C_4>0$ is independent of $r_1$ and $r_2$.  For $L>0$ and $\alpha_0>0$, we denote $u_{\lambda,\rho}=\min\{|u_\lambda|^{\alpha_0}, L^{\alpha_0}\}\rho^2u_\lambda$.  Then it is easy to see that $u_{\lambda,\rho}(x)\in E_\lambda$ for $\lambda\geq\Lambda_3$.  Since $u_\lambda$ is a solution of $(\overline{\mathcal{S}}_{\lambda,\delta_0})$ for $\lambda\geq\Lambda_3$, we have
\begin{equation*}
\int_{\bbr^3}\nabla u_\lambda\nabla u_{\lambda,\rho}+V_\lambda(x)u_\lambda u_{\lambda,\rho}+\phi_{u_{\lambda},\delta_0}g_{\delta_0}(x,u_{\lambda})u_{\lambda,\rho}dx=\int_{\bbr^3}f(x,u_\lambda)u_{\lambda,\rho}dx.
\end{equation*}
By the conditions $(A_2)$ and $(B_2)$ and the Young inequality, there exists $\Lambda_4\geq\Lambda_3$ such that
\begin{eqnarray*}
&&\int_{\bbr^3}\nabla u_\lambda\nabla u_{\lambda,\rho}+V_\lambda(x)u_\lambda u_{\lambda,\rho}+\phi_{u_{\lambda},\delta_0}g_{\delta_0}(x,u_{\lambda})u_{\lambda,\rho}dx\\
&\geq&\frac12\int_{\bbr^3}\min\{|u_\lambda|^{\alpha_0}, L^{\alpha_0}\}\rho^2|\nabla u_\lambda|^2dx-8\int_{\bbr^3}|\nabla \rho|^2|u_\lambda|^{\alpha_0+2}dx\\
&&-b_0(1+\sqrt{a_\infty})\int_{\bbr^3}|u_\lambda|^{\alpha_0+2}dx
\end{eqnarray*}
for $\lambda\geq\Lambda_4$.
This together with the condition $(F_1)$ and the choice of $r_1,r_2,r_*$, implies
\begin{equation*}
\int_{\bbr^3}\min\{|u_\lambda|^{\alpha_0}, L^{\alpha_0}\}\rho^2|\nabla u_\lambda|^2dx\leq 2f_1\int_{\bbr^3}\rho^2|u_\lambda|^{p+\alpha_0}dx+\frac{16C_4^2+2}{(r_1-r_2)^2}\int_{\bbr^3}\rho^2|u_\lambda|^{\alpha_0+2}dx.
\end{equation*}
By the Levi and Sobolev embedding theorems, we obtain
\begin{equation}\label{eq0016}
(\int_{B_{r_2}(y)}|u_\lambda|^{3(\alpha_0+1)}dx)^{\frac13}\leq(\frac{\alpha_0+2}{2})^2
(2f_1\int_{B_{r_1}(y)}|u_\lambda|^{p+\alpha_0}dx+\frac{16C_4^2+2}{(r_1-r_2)^2}\int_{B_{r_1}(y)}|u_\lambda|^{2+\alpha_0}dx).
\end{equation}
Let $r_n=(1+k^{-n})r_*$ and $\alpha_n=3(\alpha_{n-1}+1)-p$ with $k^2<2$ and $\alpha_0=\frac32$.  Then we can replace $r_1,r_2,\alpha_0$ in \eqref{eq0016} by $r_{n-1},r_n,\alpha_{n-1}$ and obtain
\begin{eqnarray}
&&(\int_{B_{r_{n+1}}(y)}|u_\lambda|^{3(\alpha_{n-1}+1)}dx)^{\frac13}\notag\\
&\leq& (\frac{\alpha_{n-1}+2}{2})^2
(2f_1\int_{B_{r_{n}}(y)}|u_\lambda|^{p+\alpha_{n-1}}dx
+\frac{16C_4^2+2}{(r_n-r_{n+1})^2}\int_{B_{r_{n}}(y)}|u_\lambda|^{2+\alpha_{n-1}}dx).\label{eq0017}
\end{eqnarray}
Clearly, one of the following two cases must occur:
\begin{enumerate}
\item[$(1)$] $\int_{B_{r_{n}}(y)}|u_\lambda|^{2+\alpha_{n-1}}dx\leq\int_{B_{r_{n}}(y)}|u_\lambda|^{p+\alpha_{n-1}}dx$ up to a subsequence.
\item[$(2)$] $\int_{B_{r_{n}}(y)}|u_\lambda|^{p+\alpha_{n-1}}dx\leq\int_{B_{r_{n}}(y)}|u_\lambda|^{2+\alpha_{n-1}}dx$ up to a subsequence.
\end{enumerate}
If case~$(1)$ happen, then by \eqref{eq0017} and the choice of $r_*$ and $r_n$, we obtain
\begin{equation}\label{eq0018}
\|u_\lambda\|^{\alpha_{n-1}+1}_{L^{3(\alpha_{n-1}+1)}(B_{r_{n+1}}(y))}\leq(\alpha_{n-1}+2)^{2}\frac{C_5}{(r_n-r_{n+1})^2}
\|u_\lambda\|_{L^{3(\alpha_{n-2}+1)}(B_{r_{n}}(y))}^{\alpha_{n-1}+p},
\end{equation}
where $C_5=16C_4^2+4$.  By iterating \eqref{eq0018}, we can see that
\begin{equation}\label{eq0019}
\|u_\lambda\|_{L^\infty(B_{r_*}(y))}\leq\bigg(\prod_{n=1}^{\infty}(\alpha_{n-1}+2)^{\frac{2}{\alpha_{n-1}+1}}(\frac{C_5}{(r_n-r_{n+1})^2})^{\frac{1}{\alpha_{n-1}+1}}
\|u_\lambda\|_{L^{\frac32+p}(B_{2r_*}(y))}\bigg)^{\prod_{n=1}^{\infty}\frac{\alpha_{n-1}+p}{\alpha_{n-1}+1}},
\end{equation}
where $\|\cdot\|_{L^\infty(B_{r_*}(y))}$ is the usual norm in $L^\infty(B_{r_*}(y))$.
Since $3<p\leq4$, we can obtain from the choice of $\alpha_n$ that $\alpha_n\geq2\alpha_{n-1}$, which implies that $\alpha_n\geq2^n$.  Note that $r_n=(1+k^{-n})r_*$ with $k^{-2}<2$, we have $\prod_{n=1}^{\infty}(\alpha_{n-1}+2)^{\frac{2}{\alpha_{n-1}+1}}(\frac{C_5}{(r_n-r_{n+1})^2})^{\frac{1}{\alpha_{n-1}+1}}<+\infty$ and $\prod_{n=1}^{\infty}\frac{\alpha_{n-1}+p}{\alpha_{n-1}+1}<+\infty$.  By \eqref{eq0019}, $\|u_\lambda\|_{L^\infty(B_{r_*}(y))}\leq C_6\|u_\lambda\|_{L^{\frac32+p}(B_{2r_*}(y))}^{C_6}$, where $C_6>0$ is independent of $\lambda\geq\Lambda_4$ and $y\in\bbr^3\backslash\Omega_a'$.  If case~$(2)$ happen, then by the H\"older inequality and a similar iteration as used in \eqref{eq0019}, we can obtain
\begin{equation}\label{eq0020}
\|u_\lambda\|_{L^\infty(B_{r_*}(y))}\leq\bigg(\prod_{n=1}^{\infty}(\alpha_{n-1}+2)^{\frac{2}{\alpha_{n-1}+1}}
(\frac{C_5|B_{r_n}|^{\frac{p-2}{\alpha_{n-1}+p}}}{(r_n-r_{n+1})^2})^{\frac{1}{\alpha_{n-1}+1}}
\|u_\lambda\|_{L^{\frac32+p}(B_{2r_*}(y))}\bigg)^{\prod_{n=1}^{\infty}\frac{\alpha_{n-1}+p}{\alpha_{n-1}+1}}.
\end{equation}
By similar arguments, we also have from \eqref{eq0020} that $\|u_\lambda\|_{L^\infty(B_{r_*}(y))}\leq C_6\|u_\lambda\|_{L^{2+p}(B_{2r_*}(y))}^{C_6}$.  Therefore, in any case, we can obtain $\|u_\lambda\|_{L^\infty(B_{r_*}(y))}\leq C_6\|u_\lambda\|_{L^{\frac32+p}(B_{2r_*}(y))}^{C_6}$, where $C_6>0$ is independent of $\lambda\geq\Lambda_4$ and $y\in\bbr^3\backslash\Omega_a'$.  Note that by Lemma~\ref{lem0005} and \eqref{eq0023}, we have $\|u_{\lambda}\|\leq(1+d_*)C^*$.  So thanks to Lemma~\ref{lem0006} and the H\"older and Sobolev inequalities, we can see that $\|u_\lambda\|_{L^\infty(B_{r_*}(y))}\to0$ as $\lambda\to+\infty$ uniformly on a.e. $\bbr^3\backslash\Omega_a'$.  It follows that there exists $\Lambda_*\geq\Lambda_2$ such that $|u_\lambda(x)|\leq\delta_0$ a.e. on $\bbr^3\backslash\Omega_a'$ for $\lambda\geq\Lambda_*$.
\end{proof}

We close this section by

\medskip\par\noindent{\bf Proof of Theorem~\ref{thm0001}:}\quad By Proposition~\ref{prop0001}, $(\overline{\mathcal{S}}_{\lambda,\delta_0})$ has a nontrivial solution $u_\lambda$ for $\lambda\geq\Lambda_3$.  It follows that $(u_\lambda,\phi_{u_\lambda,\delta_0})$ is a nontrivial solution of $(\overline{\mathcal{SP}}_{\lambda,\delta_0})$ for $\lambda\geq\Lambda_3$.  Thanks to Lemma~\ref{lem0007} and the construction of $g_{\delta_0}(x,t)$, if $3<p\leq4$ and $\lambda\geq\Lambda_*$, then $(u_\lambda,\phi_{u_\lambda,\delta_0})$ is also a nontrivial solution of $(\mathcal{SP}_{\lambda})$.  It remains to show the concentration behavior of $(u_\lambda,\phi_{u_\lambda,\delta_0})$.  Suppose $(u_{\lambda_n},\phi_{u_{\lambda_n},\delta_0})$ is a sequence of solutions with $\lambda_n\to+\infty$.  Then $u_{\lambda_n}$ is also a solution of $(\overline{\mathcal{S}}_{\lambda_n,\delta_0})$, by similar arguments as used in Lemma~\ref{lem0006}, $u_{\lambda_n}\rightharpoonup u_0$ weakly in $\h$ for some $u_0\in H^1_0(\Omega_a)$ satisfying $u_0\equiv0$ outside $\Omega_a$ and $\int_{\bbr^3\backslash\Omega_a''}u_{\lambda_n}^2dx\to0$ as $n\to\infty$.  It follows from the condition $(A_2)$ and the choice of $\Omega_a''$ that $u_{\lambda_n}\to u_0$ strongly in $L^2(\bbr^3)$ as $n\to\infty$.  Note that the condition $(F_1)$ holds and $u_{\lambda_n}\rightharpoonup u_0$ weakly in $\h$ as $n\to\infty$, by the H\"older and Sobolev inequalities, we have $\int_{\bbr^3}f(x,u_{\lambda_n})u_{\lambda_n}-f(x,u_0)u_0dx\to0$ as $n\to\infty$, which together with the conditions $(A_1)$, $(B_3)$ and $(K_1)$ and the Fatou lemma, implies
\begin{eqnarray}
&&\int_{\bbr^3}|\nabla u_0|^2+b(x)u_0^2dx+\int_{\bbr^3}\int_{\bbr^3}\frac{K(x)K(y)u^2_0(y)u_0^2(x)}{4\pi|x-y|}dydx\notag\\
&\leq&\liminf_{n\to\infty}(\mathcal{D}_{\lambda_n}(u_{\lambda_n}, u_{\lambda_n})+\int_{\bbr^3}K(x)\phi_{u_{\lambda_n},\delta_0}u^2_{\lambda_n}dx)\notag\\
&=&\liminf_{n\to\infty}\int_{\bbr^3}f(x,u_{\lambda_n})u_{\lambda_n}dx\notag\\
&=&\int_{\bbr^3}f(x,u_0)u_0dx.\label{eq0021}
\end{eqnarray}
Since $u_0\in H^1_0(\Omega_a)$, it is easy to show that $u_0$ is a solution of \eqref{eq0022}.
Therefore, by \eqref{eq0021}, we must have $\mathcal{D}_{\lambda_n}(u_{\lambda_n}, u_{\lambda_n})\to\int_{\bbr^3}|\nabla u_0|^2+b(x)u_0^2dx$ as $n\to\infty$.  It follows that $u_{\lambda_n}\to u_0$ strongly in $\h$ as $n\to\infty$.  Since $c_{\lambda_n}\in[\frac{p-2}{p}\bigg(\frac{2d_*\ve_0}{pC_*}\bigg)^{\frac{2}{p-2}}, \frac12R_0^2+\frac{p-2}{p}\bigg(\frac{2d_*\ve_0}{pC_*}\bigg)^{\frac{2}{p-2}}+f_1(d_*R_0^2+d_*^{\frac{6-p}{4}}S^{-\frac{3p-6}{4}}R_0^p)]$ and $\frac{p-2}{p}\bigg(\frac{2d_*\ve_0}{pC_*}\bigg)^{\frac{2}{p-2}}$ and $\frac12R_0^2+\frac{p-2}{p}\bigg(\frac{2d_*\ve_0}{pC_*}\bigg)^{\frac{2}{p-2}}+f_1(d_*R_0^2+d_*^{\frac{6-p}{4}}S^{-\frac{3p-6}{4}}R_0^p)$ are independent of $n$, we also have $u_0\not=0$.  Note that $\{\lambda_n\}$ is arbitrary, so $u_\lambda$ has the concentration behaviors described as in this theorem.
\qquad\raisebox{-0.5mm}{\rule{1.5mm}{4mm}}\vspace{6pt}

\section{Acknowledgements}
This work was partly completed when Y. Wu was visiting National University of Kaohsiung, and he is grateful to the members in the department of applied mathematics at National University of Kaohsiung for their invitation and hospitality.  J. Sun was supported by the National Natural Science Foundation of China
(Grant No. 11201270, No.11271372), Shandong Natural Science Foundation (Grant No. ZR2012AQ010) and Young Teacher Support Program of Shandong
University of Technology.  T.-F. Wu was supported in part by the National Science Council and the National Center for Theoretical Sciences, Taiwan.  Y. Wu was supported by the Fundamental Research Funds for the Central Universities (2014QNA67).

\end{document}